\documentclass[article,12pt]{elsarticle}

	\usepackage[margin=1.5in]{geometry}  
	\usepackage{graphicx}              
	\usepackage{amsmath} 
	\usepackage{orcidlink}              
	\usepackage{amsfonts}
	\usepackage{amsthm} 
	\usepackage{colortbl}               
	\usepackage{framed}
	\usepackage{comment}
	\usepackage{url}
	
	\newtheorem{thm}{Theorem}[section]
	\newtheorem{lem}[thm]{Lemma}
	\newtheorem{prop}[thm]{Proposition}

	\makeatletter
\def\ps@pprintTitle{%
  \let\@oddhead\@empty
  \let\@evenhead\@empty
  \def\@oddfoot{\reset@font\hfil}
  \def\@evenfoot{\reset@font\hfil}
}
\makeatother
\begin{document}

\begin{frontmatter}

\title{\bf On Proving Ramanujan's Inequality using a Sharper Bound for the Prime Counting Function $\pi(x)$}

\author[affil]{\textsc{Subham De} \orcidlink{0009-0001-3265-4354}}

\address[affil]{Department of Mathematics, Indian Institute of Technology Delhi, India \footnote{email: subham581994@gmail.com}\footnote{Website: \url{www.sites.google.com/view/subhamde}}}

\begin{abstract}
\noindent This article provides a proof that the Ramanujan's Inequality given by, 
			
$$\pi(x)^2 < \frac{e x}{\log x} \pi\Big(\frac{x}{e}\Big)$$
holds unconditionally for every $x\geq \exp(43.5102147)$. In case for an alternate proof of the result stated above, we shall exploit certain estimates involving the Chebyshev Theta Function, $\vartheta(x)$ in order to derive appropriate bounds for $\pi(x)$, which'll lead us to a much improved condition for the inequality proposed by Ramanujan to satisfy unconditionally.
\end{abstract}


\end{frontmatter}

\section{Introduction}
The notion of analyzing the proportion of prime numbers over the real line $\mathbb{R}$ first came into the limelight thanks to the genius work of one of the greatest and most gifted mathematicians of all time named \textit{Srinivasa Ramanujan}, as evident from his letters \cite[pp. xxiii-xxx , 349-353]{12} to another one of the most prominent mathematicians of $20^{th}$ century, \textit{G. H. Hardy} during the months of Jan/Feb of $1913$, which are testaments to several strong assertions about the \textit{Prime Counting Function}, $\pi(x)$ [cf. Definition $(2.3)$ \cite{15}].\par 
In the following years, Hardy himself analyzed some of those results \cite{13} \cite[pp. 234-238]{14}, and even wholeheartedly acknowledged them in many of his publications, one such notable result is the \textit{Prime Number Theorem} [cf. Theorem $(2.4)$ \cite{15}].\par 
\textit{Ramanujan} provided several inequalities regarding the behavior and the asymptotic nature of $\pi(x)$. One of such relation can be found in the notebooks written by Ramanujan himself has the following claim.
\begin{thm}\label{thm2}
	(Ramanujan's Inequality \cite{1})  For $x$ sufficiently large, we shall have,
	\begin{align}\label{1}
		(\pi(x))^{2}<\frac{ex}{\log x}\pi\left(\frac{x}{e}\right)
	\end{align}
\end{thm}
Worth mentioning that, Ramanujan indeed provided a simple, yet unique solution in support of his claim. Furthermore, it has been well established that, the result is not true for every positive real $x$. Thus, the most intriguing question that the statement of Theorem \eqref{thm2} poses is, \textit{is there any $x_{0}$ such that, Ramanujan's Inequality will be unconditionally true for every $x\geq x_{0}$}?

A brilliant effort put up by \textit{F. S. Wheeler, J. Keiper, and W. Galway} in search for such $x_{0}$ using tools such as \texttt{MATHEMATICA} went in vain, although independently \textit{Galway} successfully computed the largest prime counterexample below $10^{11}$ at $x = 38\mbox{ }358\mbox{ }837\mbox{ }677$. However, \textit{Hassani} \cite[Theorem 1.2]{3} proposed a more inspiring answer to the question in a way that, $\exists$ such $x_{0}=138\mbox{ }766\mbox{ } 146\mbox{ } 692\mbox{ } 471\mbox{ } 228$ with \eqref{1} being satisfied for every $x\geq x_{0}$, but \textit{one has to neccesarily assume the Riemann Hypothesis}. In a recent paper by \textit{A. W. Dudek} and \textit{D. J. Platt} \cite[Theorem 1.2]{2}, it has been established that, ramanujan's Inequality holds true unconditionally for every $x\geq \exp(9658)$. Although this can be considered as an exceptional achievement in this area, efforts of further improvements to this bound are already underway. For instance, \textit{Mossinghoff} and \textit{Trudgian} \cite{5} made significant progress in this endeavour, when they established a better estimate as, $x\geq \exp(9394)$. Later on, \textit{Platt} and \textit{Trudgian} \cite[cf. Th. 2]{18} together established that, further improvement is indeed possible, and that $x\geq \exp(3915)$. Worth mentioning that, \textit{Cully-Hugill and Johnston} \cite[cf. Cor. 1.6]{19} literally took it to the next level by obtaining an effective bound for \eqref{1} to hold unconditionally as, $x\geq \exp(3604)$. Unsurprisingly, \textit{Johnston and Yang} \cite[cf. Th. 1.5]{20} outperformed them in claiming the lower bound for such $x$ satisfying \textit{Ramanujan's Inequality} to be $\exp(3361)$. \par 
 One recent even better result by Axler \cite{6} suggests that, the lower bound for $x$, namely $\exp(3361)$ can in fact be further improved upto $\exp(3158.442)$ using similar techniques as described in \cite{2}, although modifying the error term accordingly adhering to a sharper bound involving $\pi(x)$ and $Li(x)$ derived by \textit{Fiori, Kadiri, and Swidinsky} \cite[cf. Cor. 22]{4}.\par 
 This paper does indeed adopts a new approach in modifying the existing estimates for $x_{0}$ in order for the \textit{Ramanujan's Inequality} (cf. Theorem \eqref{thm2}) to hold without imposing any further assumptions on it for every $x\geq x_{0}$. By utilizing some effective bounds on the \textit{Chebyshev's $\vartheta$-function}, the primary intention is to obtain a suitable bound for $\pi(x)$, and hence eventually come up with a much better estimate for $x_{0}$ by tinkering with the constants while respecting all the stipulated conditions available to us.
		\section{An Improved Criterion for Ramanujan's Inequality}
		
		Suppose, we define, 
		\begin{align}\label{10}
			\mathcal{G}(x):=(\pi(x))^2-\frac{ex}{\log x}\pi\left(\frac{x}{e}\right)
		\end{align}
		A priori using the \textit{Prime Number Theorem} \cite[cf. Th. $(2.4)$]{15}, we can in fact assert that \cite{2},
		\begin{align}\label{11}
			\pi(x)=x\sum\limits_{k=0}^{4}\frac{k!}{\log^{k+1} x}+O\left(\frac{x}{\log^6 x}\right)
		\end{align}
		as $x\rightarrow \infty$. On the other hand, for the \textit{Chebyshev's $\vartheta$-function} having the following definition,
		\begin{align}\label{2}
			\vartheta(x):=\sum\limits_{p\leq x}\log p \mbox{ , }
		\end{align}  
		we can indeed summarize certain inequalities (cf. \cite{7} and \cite{8}) as follows:
		\begin{prop}\label{prop1}
			The following holds true for $\vartheta(x)$:
			\begin{enumerate}
				\item $\vartheta(x)<x$, \hspace{100pt} for $x<10^{8}$,
				\item $|\vartheta(x)-x|<2.05282\sqrt{x}$, \hspace{25pt} for $x<10^{8}$,
				\item $\left| \vartheta(x)-x\right|<0.0239922\frac{x}{\log x}$, \hspace{10pt} for $x\geq 758711$,
				\item $|\vartheta(x)-x|<0.0077629\frac{x}{\log x}$, \hspace{10pt} for $x\geq \exp(22)$,
				\item $|\vartheta(x)-x|<8.072\frac{x}{\log^2 x}$, \hspace{30pt} for $x>1$.
			\end{enumerate}
		\end{prop}
		Applying these inequalities, we can compute a suitable bound for $\vartheta(x)$ as follows:
		\begin{lem}[cf. \cite{9}]\label{lemma1}
			We shall have the following estimate for $\vartheta(x)$:
			\begin{align}\label{12}
				x\left(1-\frac{2}{3(\log x)^{1.5}}\right)<\vartheta(x)<x\left(1+\frac{1}{3(\log x)^{1.5}}\right)\mbox{ , }\hspace{20pt} \mbox{ for }x\geq 6400.
			\end{align}
		\end{lem}
		Lemma \eqref{lemma1} does in fact enable us deduce a more effective bound for $\pi(x)$, which'll prove to be immensely beneficial for us later on.
		\begin{thm}\label{thm1}
			We shall have the following estimate for $\pi(x)$ as follows:
			\begin{align}\label{13}
				\frac{x}{\log x -1+\frac{1}{\sqrt{\log x}}}<\pi(x)<\frac{x}{\log x -1-\frac{1}{\sqrt{\log x}}}\mbox{ , }\hspace{20pt} \mbox{ for }x\geq 59.
			\end{align}
		\end{thm}
		We briefly discuss the proof of the Theorem above following the steps as described in \cite{9} for the convenience of our readers.
		\begin{proof}
			Applying a well-known inequlity involving $\vartheta(x)$ and $\pi(x)$,
			\begin{align}
				\pi(x)=\frac{\vartheta(x)}{\log x}+\int\limits_{2}^{x}\frac{\vartheta(t)}{t\log^2 t}dt
			\end{align}
			and, with the help of \eqref{12} in Lemma \eqref{lemma1}, we get,
			\begin{align*}
				\pi(x)<\frac{x}{\log x}+\frac{x}{3(\log x)^{2.5}}+\int\limits_{2}^{x}\frac{dt}{\log^2 t}+\frac{1}{3}\int\limits_{2}^{x}\frac{dt}{(\log x)^{3.5}}
			\end{align*}
			\begin{align*}
				\hspace{20pt}=\frac{x}{\log x}\left(1+\frac{1}{3(\log x)^{1.5}}+\frac{1}{\log x}\right)-\frac{2}{\log^2 2}+2\int\limits_{2}^{x}\frac{dt}{\log^3 t}+\frac{1}{3}\int\limits_{2}^{x}\frac{dt}{(\log t)^{3.5}}
			\end{align*}
			\begin{align}\label{15}
				<\frac{x}{\log x}\left(1+\frac{1}{3(\log x)^{1.5}}+\frac{1}{\log x}\right)+\frac{7}{3}\int\limits_{2}^{x}\frac{dt}{\log^3 t}
			\end{align}
			Moreover, defining the function,
			\begin{align}\label{14}
				h_{1}(x):=\frac{2}{3}.\frac{x}{(\log x)^{2.5}}-\frac{7}{3}\int\limits_{2}^{x}\frac{dt}{\log^3 t}\mbox{ , }\hspace{20pt} \mbox{ for }x\geq \exp(18.25)
			\end{align}
			We can observe that, $h_{1}'(x)>0$, implying that, $h_{1}$ is increasing. Now, for every convex function $u\mbox{ }: \mbox{ }[a,b]\rightarrow \mathbb{R}$, where, $a<b \mbox{ , }a,b\in \mathbb{R}_{>0}$, we have, 
			\begin{align}\label{16}
				\int\limits_{a}^{b}u(x)dx\leq \frac{b-a}{n}\left(u(a)+u(b)+\sum\limits_{k=1}^{n-1}u\left(a+k\frac{b-a}{n}\right)\right).
			\end{align}
			Thus, choosing $u(x):=\frac{1}{\log^3 x}$ and $n=10^5$ and using \eqref{16} on each of the intervals $[2,e]$, $[e,e^2]$, ......, $[e^{17},e^{18}]$ and $[e^{18},e^{18.25}]$ yields,
			\begin{align*}
				\int\limits_{2}^{\exp(18.25)}\frac{dt}{\log^3 t}<16870.
			\end{align*}
			Furthermore, one can also verify using \texttt{MATHEMATICA} that,
			\begin{align*}
				h_{1}(\exp(18.25))>\frac{1}{3}(118507-118090)>0.
			\end{align*}
			Therefore, for every $x\geq \exp(18.25)$, we must have from \eqref{15},
			\begin{align}\label{17}
				\pi(x)<\frac{x}{\log x}\left(1+\frac{1}{3(\log x)^{1.5}}+\frac{1}{\log x}\right)<\frac{x}{\log x -1-\frac{1}{\sqrt{\log x}}}
			\end{align}
			Again, for $x\leq \exp(18.25)<10^8$, we apply $(1)$ in Proposition \eqref{prop1} to derive,
			\begin{align*}
				\pi(x)=\frac{\vartheta(x)}{\log x}+\int\limits_{2}^{x}\frac{\vartheta(x)}{t\log^2 t}dt<\frac{x}{\log x}+\int\limits_{2}^{x}\frac{dt}{\log^2 t}
			\end{align*}
			\begin{align*}
				\hspace{150pt}=\frac{x}{\log x}\left(1+\frac{1}{\log x}\right)-\frac{2}{\log^2 2}+2\int\limits_{2}^{x}\frac{dt}{\log^3 t}.
			\end{align*}
			Furthermore, for $4000\leq x< 10^8$, taking the function,
			\begin{align}
				h_{2}(x):=\frac{x}{(\log x)^{2.5}}-2\int\limits_{2}^{x}\frac{dt}{\log^3 t}+\frac{2}{\log^2 2}.
			\end{align}
			We can indeed verify that, $h_{2}'(x)>0$, implying $h_{2}$ is an increasing function. Similarly, with the help of \texttt{MATHEMATICA}, we can compute the sign of $h_{2}$ as follows,
			\begin{align*}
				h_{2}(\exp(11))>149-2\int\limits_{2}^{\exp(11)}\frac{dt}{\log^3 t}>149-140>0.
			\end{align*}
			In summary, thus for $\exp(11)\leq x<10^8$, 
			\begin{align}\label{18}
				\pi(x)<\frac{x}{\log x}\left(1+\frac{1}{\log x}+\frac{1}{(\log x)^{1.5}}\right)<\frac{x}{\log x -1-\frac{1}{\sqrt{\log x}}}.
			\end{align}
			In addition to the above, it is important to note that, for $x\geq 6$, the denominator, $\log x -1-\frac{1}{\sqrt{\log x}}>0$. Which means that, for $6\leq x \leq \exp(11)$, we need to establish,
			\begin{align}\label{19}
				H(x):=\frac{x}{\pi(x)}+1+(\log x)^{-0.5}-\log x>0.
			\end{align}
			Assuming $p_{n}$ to be the $n^{th}$ prime, it can be observed that, $H$ is in fact increasing in $\left[p_{n},p_{n+1}\right)$, thus it only needs to be proven that, $H(p_{n})>0$. \par 
			For $p_{n}<\exp(11)$, we have the inequality $\frac{1}{\sqrt{\log p_{n}}}>0.3$, which reduces our computation to verifying,
			\begin{align*}
				\frac{p_{n}}{n}-\log p_{n}>-1.3
			\end{align*}
			for every $7\leq p_{n}\leq \exp(11)$, which can be achieved using \texttt{MATHEMATICA}.\par 
			In order to establish the lower bound of $\pi(x)$ as claimed in \eqref{13}, we shall be needing $(1)$ in Proposition \eqref{prop1} and \eqref{12} in Lemma \eqref{lemma1} under the condition that, $x\geq 6400$. Hence,
			\begin{align}\label{20}
				\pi(x)-\pi(6400)=\frac{\vartheta(x)}{\log x}-\frac{\vartheta(6400)}{\log (6400)}+\int\limits_{6400}^{x}\frac{\vartheta(t)}{t\log^2 t}dt.
			\end{align}
			Rigorous computations does yield, $\pi(6400)=834$, and, $\frac{\vartheta(6400)}{\log (6400)}<\frac{6400}{\log (6400)}<731$. Thus, \eqref{20} further reduces to,
			\begin{align*}
				\pi(x)>103+\frac{\vartheta(x)}{x}+\int\limits_{6400}^{x}\frac{\vartheta(t)}{t\log^2 t}dt.
			\end{align*}
			Using the lower bound of $\vartheta(x)$ as in \eqref{12} of Lemma \eqref{lemma1} gives,
			\begin{align*}
				\pi(x)>103+\frac{x}{\log x}-\frac{2x}{3\log^{2.5} x}+\frac{x}{\log^2 x}-\frac{6400}{\log^2 6400}+2\int\limits_{6400}^{x}\frac{dt}{\log^3 t}-\frac{2}{3}\int\limits_{6400}^{x}\frac{dt}{\log^{3.5} t}
			\end{align*}
			\begin{align*}
				\hspace{100pt}>\frac{x}{\log x}\left(1+\frac{1}{\log x}-\frac{2}{3\log^{1.5} x}\right)>\frac{x}{\log x -1+\frac{1}{\sqrt{\log x}}}
			\end{align*}
			Setting $v=(\log x)^{-0.5}$, we can assert that, the above inequality holds true for, $2v^3 -5v^2 +3v-1<0$, implying, $v(1-v)(3-2v)\leq \frac{(3-v)}{4}<1$. Hence, it can be confirmed that, the statement \eqref{13} holds true for $x\geq 6400$.\par 
			Furthermore, for $x<6400$, we intend on showing that,
			\begin{align}
				\beta(x):=-\frac{x}{\pi(x)}+\log x -1 +\frac{1}{\sqrt{\log x}}>0.
			\end{align}
			Assuming similarly that, $p_{n}$ denotes the $n^{th}$ prime, one can observe that, the function $\beta(x)$ is indeed decreasing on $\left[p_{n},p_{n+1}\right)$. Hence, it only suffices to check for the values at $p_{n}-1$. Now, $p_{n}\leq 6400$ implies, $(\log (p_{n}-1))^{-0.5}>0.337$, and thus, it only is needed to be checked that,
			\begin{align}\label{21}
				\frac{\log (p_{n}-1)}{p_{n}-1}-\frac{p_{n}-1}{n-1}>0.663
			\end{align}
			Utilizing proper coding in \texttt{MATHEMATICA} gives us, $n\geq 36$ in order for \eqref{21} to satisfy. Therefore, we can further verify that, \eqref{13} holds for $x\geq 59$, and the proof is complete.
		\end{proof}
		Significantly, Karanikolov \cite{10} cited one of the applications of \eqref{13} which says that for  $\alpha\geq e^{1/4}$ and, $x\geq 364$, we must have,
		\begin{align}\label{22}
			\pi(\alpha x)<\alpha \pi(x).
		\end{align}
		Although, a more effective version of \eqref{22} states (cf. Theorem 2 \cite{9}) the following.
		\begin{prop}\label{prop2}
			\eqref{22} holds true for every $\alpha>1$ and, $x>\exp\left(4(\log \alpha)^{-2}\right)$,
		\end{prop}
		\begin{proof}
			We utilize \eqref{13} in theorem \eqref{thm1} for $\alpha x\geq 6$. Thus,
			\begin{align*}
				\frac{\alpha x}{\log \alpha x -1+\frac{1}{\sqrt{\log \alpha x}}}<\pi(\alpha x)<\frac{\alpha x}{\log \alpha x -1-\frac{1}{\sqrt{\log \alpha x}}}
			\end{align*}
			and,
			\begin{align*}
				\frac{\alpha x}{\log x -1+\frac{1}{\sqrt{\log x}}}<\alpha \pi(x)<\frac{\alpha x}{\log x -1-\frac{1}{\sqrt{\log x}}}
			\end{align*}
			for every $x\geq 59$. Now, assuming $x\geq \exp\left(4(\log \alpha)^{-2}\right)$, we can deduce that,
			\begin{align*}
				\log \alpha >(\log \alpha x)^{-0.5}+(\log x)^{-0.5}
			\end{align*}
			which is all that we're required to show. This completes the proof.
		\end{proof}
		As for another application of \eqref{13}, we must mention the work of Udrescu \cite{11}, where it was claimed that, if $0<\epsilon\leq 1$, then,
		\begin{align}\label{23}
			\pi(x+y)<\pi(x)+\pi(y) \mbox{ , }\hspace{20pt}\forall \mbox{ }\epsilon x\leq y\leq x.
		\end{align}
		Again, further progress have in fact been made in order to improve the result \eqref{23}. One such notable work in this regard has been done by Panaitopol \cite{9}.
		\begin{lem}
			\eqref{23} is satisfied under additional condition, $x\geq \exp(9\epsilon^{-2})$, where, $\epsilon\in \left(0,1\right]$.
		\end{lem}
		We shall be using all the above derivations in order to obtain a much improved bound for $x_{0}$ such that, $\mathcal{G}(x)<0$ unconditionally for every $x\geq x_{0}$.\par 
		
		Choose some $a>1$ such that, $e-a>a>1$ as well. Hence,
		\begin{align}\label{24}
			\pi(x)=\pi\left(e.\frac{x}{e}\right)=\pi\left(a.\frac{x}{e}+(e-a).\frac{x}{e}\right)
		\end{align}
		Using \eqref{23} by taking, $\epsilon=\frac{a}{e-a}<1$ as per our construction yields,
		\begin{align}\label{25}
			\pi(x)<\pi\left(a.\frac{x}{e}\right)+\pi\left((e-a).\frac{x}{e}\right)
		\end{align}
		for every $x\geq \frac{e}{e-a}.\exp\left(9.\left(\frac{a}{e-a}\right)^{-2}\right)$.
		Furthermore, by our selection of $a$, we can in fact utilize Proposition \eqref{prop2} again in order to derive the following estimates,
		\begin{align}\label{26}
			\pi\left(a.\frac{x}{e}\right)<a.\pi\left(\frac{x}{e}\right) \mbox{ , } \forall \mbox{ }x>\exp\left(4(\log a)^{-2}+1\right)
		\end{align}
		and,
		\begin{align}\label{27}
			\pi\left((e-a).\frac{x}{e}\right)<(e-a).\pi\left(\frac{x}{e}\right)\mbox{ , } \forall \mbox{ }x>\exp\left(4(\log (e-a))^{-2}+1\right)
		\end{align}
		Therefore, combining \eqref{24}, \eqref{25}, \eqref{26} and \eqref{27}, we obtain,
		\begin{align}\label{28}
			\pi(x)<a.\pi\left(\frac{x}{e}\right)+(e-a).\pi\left(\frac{x}{e}\right)=e.\pi\left(\frac{x}{e}\right)
		\end{align}
		for every such,
		\begin{align}\label{29}
			x> \max\left\{\frac{e}{e-a}.\exp\left(9.\left(\frac{a}{e-a}\right)^{-2}\right),\exp\left(4(\log a)^{-2}+1\right),\exp\left(4(\log (e-a))^{-2}+1\right)\right\}
		\end{align}
		For our convenience, we consider, $a=1.359>1$. 
		
		Thus, we can verify, 
		$$e-a=1.359281828>a>1 \mbox{ and, } \epsilon=0.999792663<1,$$ as desired. Subsequently, we conclude that, \eqref{28} is satisfied for every,
		\begin{align}
			x> \max\left\{1.999792664\exp(9.003733214),\exp(43.5102146),\exp(43.45280029)\right\}
		\end{align}
		In summary, we have, 
		\begin{align}\label{30}
			\pi(x)<e.\pi\left(\frac{x}{e}\right)\mbox{ , }\forall \mbox{ }x\ge \exp(43.5102147).
		\end{align}
		On the other hand, \eqref{11} gives us,
		\begin{align}\label{31}
			\pi(x)>\frac{x}{\log x}
		\end{align}
		for sufficiently large values of $x$. In fact, one can verify numerically using \texttt{MATHEMATICA} that, \eqref{31} holds true for every $x\geq \exp(43)$. Finally, combining \eqref{30} and \eqref{31}, we get from \eqref{10},
		\begin{align}\label{32}
			\mathcal{G}(x)=(\pi(x))^2 +\left(\frac{x}{\log x}\right).\left(-e.\pi\left(\frac{x}{e}\right)\right)<(\pi(x))^2+\pi(x).(-\pi(x))=0.
		\end{align}
		and this is valid unconditionally for every $x\geq \exp(43.5102147)$. Therefore, we have our $x_{0}=\exp(43.5102147)$ as desired in order for the \textit{Ramanujan's Inequality} to hold without any further assumptions.
		\section{Numerical Estimates for $\mathcal{G}(x)$}
		We can indeed verify our claim using programming tools such as \texttt{MATHEMATICA} for example. The numerical data\footnote[1]{Codes are available at: \url{https://github.com/subhamde1/Paper-15.git}} from the Table \eqref{table 1} and the plot \eqref{fig1} representing values of $\log (-\mathcal{G}(x))$ with respect to $\log x$ for $x\in[\exp(43),\exp(3159)]$ clearly establishes that, $\mathcal{G}$ is indeed \textbf{monotone decreasing} on the interval $[\exp(43.5102147),\exp(3159)]$ and also is \textbf{strictly negative}. It only suffices to check until $\exp(3159)$, as the result has been unconditionally proven for $x\geq \exp(3158.442)$ by Axler \cite{6}. 
				
		\begin{figure}[hbt!]
			\centering
			\includegraphics[width=0.9\linewidth]{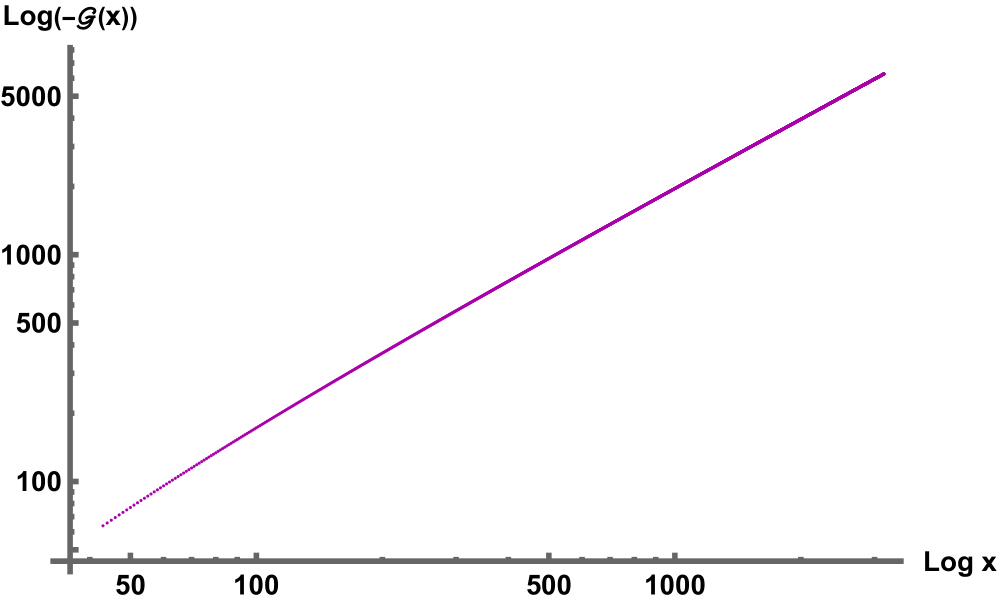}
			\caption{Plot of $\log (-\mathcal{G}(x))$ with respect to $\log x$}
			\label{fig1}
		\end{figure}		
		
		\begin{table}[hbt!]
			\centering
			\begin{tabular}{|c|c|}
				\hline
					\rowcolor{gray}
				\textbf{\( x \)} & \textbf{\( \mathcal{G}(x) \) }\\
				\hline
				$  $ & $  $ \\
			$e^{43.5102147}$ & $-1.2984816 \times 10^{28}$ \\
			$e^{49}$ & $-3.5777143 \times 10^{32}$ \\
			$e^{59}$ & $-5.3863026 \times 10^{40}$ \\
			$e^{159}$ & $-8.6366147 \times 10^{124}$ \\
			$e^{259}$ & $-3.2250049 \times 10^{210}$ \\
			$e^{359}$ & $-3.2357043 \times 10^{296}$ \\
			$e^{459}$ & $-5.3064365 \times 10^{382}$ \\
			$e^{559}$ & $-1.1686993 \times 10^{469}$ \\
			$e^{659}$ & $-3.1339236 \times 10^{555}$ \\
			$e^{759}$ & $-9.6742945 \times 10^{641}$ \\
			$e^{859}$ & $-3.3194561 \times 10^{728}$ \\
			$e^{959}$ & $-1.2367077 \times 10^{815}$ \\
			$e^{1059}$ & $-4.9214899 \times 10^{901}$ \\
			$e^{1159}$ & $-2.0671392 \times 10^{988}$ \\
			$e^{1259}$ & $-9.0822473 \times 10^{1074}$ \\
			$e^{1359}$ & $-4.1454353 \times 10^{1161}$ \\
			$e^{1459}$ & $-1.9549848 \times 10^{1248}$ \\
				\hline
			\end{tabular}
			\quad
			\begin{tabular}{|c|c|}
				\hline
					\rowcolor{gray}
				\textbf{\( x \)} & \textbf{\( \mathcal{G}(x) \) }\\
				\hline
					$  $ & $  $ \\
				$e^{1559}$ & $-9.4847597 \times 10^{1334}$ \\
				$e^{1659}$ & $-4.7172079 \times 10^{1421}$ \\
				$e^{1759}$ & $-2.3980349 \times 10^{1508}$ \\
				$e^{1859}$ & $-1.2430367 \times 10^{1595}$ \\
				$e^{1959}$ & $-6.5566576 \times 10^{1681}$ \\
				$e^{2059}$ & $-3.5131458 \times 10^{1768}$ \\
				$e^{2159}$ & $-1.9093149 \times 10^{1855}$ \\
				$e^{2259}$ & $-1.0511565 \times 10^{1942}$ \\
				$e^{2359}$ & $-5.8557034 \times 10^{2028}$ \\
				$e^{2459}$ & $-3.2975152 \times 10^{2115}$ \\
				$e^{2559}$ & $-1.8754944 \times 10^{2202}$ \\
				$e^{2659}$ & $-1.0765501 \times 10^{2289}$ \\
				$e^{2759}$ & $-6.2322859 \times 10^{2375}$ \\
				$e^{2859}$ & $-3.6365683 \times 10^{2462}$ \\
				$e^{2959}$ & $-2.1376236 \times 10^{2549}$ \\
				$e^{3059}$ & $-1.2651826 \times 10^{2636}$ \\
				$e^{3159}$ & $-7.5364298 \times 10^{2722}$ \\
				\hline
			\end{tabular}
			\caption{Values of \( \mathcal{G}(x) \)}
			\label{table 1}
		\end{table}

\clearpage 
\section{Future Research Prospects}
In summary, we've utilized specific order estimates for the \textit{Prime Counting Function} $\pi(x)$ in addition to several explicit bounds involving \textit{Chebyshev's $\vartheta$-function}, $\vartheta(x)$, a priori with the help of the \textit{Prime Number Theorem} in order to conjure up an improved bound for the famous \textit{Ramanujan's Inequality}. Although, it'll surely be interesting to observe whether it's at all feasible to apply any other techniques for this purpose.\par 
On the other hand, one can surely work on some modifications of \textit{Ramanujan's Inequality} For instance, \textit{Hassani} studied \eqref{1} extensively for different cases \cite{3}, and eventually claimed that, the inequality does in fact reverses if one can replace $e$ by some $\alpha$ satifying, $0<\alpha<e$, although it retains the same sign for every $\alpha \geq e$.\par 
In addition to above, it is very much possible to come up with certain generalizations of Theorem \eqref{thm2}. In this context, we can study \textit{Hassani}'s stellar effort in this area where, he apparently increased the power of $\pi(x)$ from $2$ upto $2^n$ and provided us with this wonderful inequality stating that for sufficiently large values of $x$ \cite{16}, 
\begin{align*}
	(\pi(x))^{2^n}<\frac{e^n}{\prod\limits_{k=1}^{n}\left(1-\frac{k-1}{\log x}\right)^{2^{n-k}}}\left(\frac{x}{\log x}\right)^{2^n -1}\pi\left(\frac{x}{e^n}\right)
\end{align*}
Finally, and most importantly, we can choose to broaden our horizon, and proceed towards studying the \textit{prime counting function} in much more detail in order to establish other results analogous to Theorem \eqref{thm2}, or even study some specific polynomial functions in $\pi(x)$ and also their powers if possible. One such example which can be found in \cite{17} eventually proves that, for sufficiently large values of $x$, 
\begin{align*}
	\frac{3ex}{\log x} \left(\pi\left(\frac{x}{e}\right)\right)^{3^n - 1}<(\pi(x))^{3^n} + \frac{3e^2 x}{(\log x)^2} \left(\pi\left(\frac{x}{e^2}\right)\right)^{3^n - 2} \mbox{ , }\hspace{10pt} n>1
\end{align*} 
Whereas, significantly the inequality reverses for the specific case when, $n=1$ (\textit{Cubic Polynomial Inequality}) (cf. Theorem $(3.1)$ \cite{17}).\par 
Hopefully, further research in this context might lead the future researchers to resolve some of the unsolved mysteries involving \textit{prime numbers}, or even solve some of the unsolved problems surrounding the iconic field of Number Theory.

\vspace{80pt}
\section*{Acknowledgments}
I'll always be grateful to \textbf{Prof. Adrian W. Dudek} ( Adjunct Associate Professor, Department of Mathematics and Physics, University of Queensland, Australia ) for inspiring me to work on this problem and pursue research in this topic. His leading publications in this area helped me immensely in detailed understanding of the essential concepts.

\section*{Statements and Declarations}
\subsection*{Conflicts of Interest Statement}
I as the author of this article declare no conflicts of interest.
\subsection*{Data Availability Statement}
I as the sole author of this article confirm that the data supporting the findings of this study are available within the article [and/or] its supplementary materials.

\bibliographystyle{elsarticle-num}

\end{document}